\documentclass[11pt]{article}
\usepackage{geometry}
\usepackage{amsmath}
\usepackage{amssymb}
\usepackage{amsthm}
\usepackage{mathtools}
\usepackage{physics}
\usepackage{relsize}
\usepackage{float}
\usepackage{graphicx}
\usepackage{subcaption}
\usepackage[justification=centering]{caption}
\usepackage{bigints}
\usepackage{import}
\usepackage{titlepic}
\usepackage[backend=biber,style=ieee]{biblatex}
\usepackage{pdfpages}
\theoremstyle{plain}
\newtheorem{theorem}{Theorem}[section]
\newtheorem{lemma}[theorem]{Lemma}
\newtheorem{proposition}[theorem]{Proposition}
\newtheorem{corollary}[theorem]{Corollary}

\newtheorem{definition}[theorem]{Definition}

\usepackage[symbol]{footmisc}

\title{Revivals in time-evolution quasi-periodic problems}
\author{George Farmakis \\
\footnotesize{Heriot-Watt University \& Maxwell Institute for Mathematical Sciences, Edinburgh, UK.} \\
\footnotesize{\emph{Email address:} G.Farmakis@hw.ac.uk}}
\begin{document}
\maketitle

\begin{abstract}
We examine the influence of quasi-periodic boundary conditions on the phenomenon of revivals in linear dispersive PDEs. We show that, in general, quasi-periodic problems do not support the revival effect at rational times. Our method is based on a correspondence between quasi-periodic and periodic problems. We prove that the solution to a quasi-periodic problem is expressed via the solution to a corresponding periodic problem, and vice-versa. Then, our main results follow by deriving a representation of the periodic problem solution in terms of a composition of solutions for a particular class of periodic problems, where the latter supports the classical revival and fractalisation dichotomy at rational and irrational times.
\end{abstract}



\section{Introduction}\label{Introduction}
In this work, we are interested in the revival effect in the following time-evolution problem, which we will call the \emph{quasi-periodic} problem, 
\begin{equation}
    \label{Quasi-Periodic Problem}
    \begin{aligned}
    &\partial_{t}u(x,t) = - i P(-i\partial_{x}) u(x,t), \quad u(x,0) = u_{0}(x) \\
    &e^{i2\pi\theta} \partial_{x}^{m} u(0,t) = \partial_{x}^{m} u(2\pi,t), \quad m = 0,1\dots, n-1.
    \end{aligned}
\end{equation}
Here $(x,t)\in [0,2\pi]\cross[0,\infty)$, $\theta\in (0,1)$, $n\in\mathbb{Z}$ such that $n\geq 3$, and $P$ is the polynomial of order $n$ given by
\begin{equation}
    \label{Polynomial P}
    P(\lambda) = \sum_{m=0}^{n} \alpha_{m} \lambda^{m}, \quad \alpha_{m} \in\mathbb{Z}, \quad \alpha_{n}\not=0.
\end{equation}
Note that the PDE in \eqref{Quasi-Periodic Problem} is a linear dispersive PDE with dispersion relation  $\omega (\lambda) = P(\lambda)$.

The phenomenon of revivals, as was shown by Olver in \cite{olver2010dispersive}, occurs in the class of linear dispersive PDEs in \eqref{Quasi-Periodic Problem} under periodic boundary conditions (that is if we allow $\theta = 0$ in \eqref{Quasi-Periodic Problem}). Plainly, this means that for a given arbitrary initial condition, the solution evaluated at times that are rational multiples of the period $(t/2\pi\in\mathbb{Q})$, called \emph{rational times}, is expressed in terms of a finite linear combination of periodic translations of the initial condition. This recurrence of the initial function in the structure of the solution at rational times is referred to as the \emph{revival effect} \cite{berry2001quantum}. In particular, due to the revival effect, when the initial function has finitely many jump discontinuities, then the solution at rational times exhibits a finite number of jump discontinuities.

The revival effect at rational times is in contrast to the behaviour of the solution at \emph{irrational times} ($t/2\pi\not\in\mathbb{Q}$). Indeed, under periodic boundary conditions, the solution at irrational times evolves from an initially discontinuous function to a continuous one. Moreover, for appropriate singular initial conditions, the profiles of the real and imaginary parts of the solution display a fractal-like behaviour. This is known as the \emph{fractalisation effect} and was examined by Chousionis, Erdo\u{g}an and Tzirakis in  \cite{chousionis2014fractal}, for the monomial case $P(\lambda) = \lambda^{n}$, $n\geq 3$, and extended later in the monograph \cite{erdougan2016dispersive} by Erdo\u{g}an and Tzirakis for general $P$ as in \eqref{Polynomial P}.

Our motivation for studying the revival effect in the quasi-periodic problem \eqref{Quasi-Periodic Problem} stems from our previous work \cite{boulton2021beyond}. In \cite{boulton2021beyond}, it was shown that at rational times the revival property holds for the quasi-periodic free linear Schr\'{o}dninger equation ($P(\lambda) = \lambda^{2}$), but, in general, it breaks down for the solution to the quasi-periodic Airy PDE ($P(\lambda) = \lambda^{3}$).  
Here, we generalise the collapse of revivals from the quasi-periodic Airy PDE \cite{boulton2021beyond} to higher-order linear dispersive PDEs (see Theorem~\ref{Main theorem}). As a consequence of the analysis, we are also able to identify a class of periodic problems that do not support revivals (see Theorem~\ref{Main theorem 2}). 

The class of quasi-periodic problems \eqref{Quasi-Periodic Problem} provide an additional model of time-evolution problems for which we can establish a universal theory of revivals for the whole family of equations. With the exception of \cite{yin2023dispersive} and \cite{farmakis2023new}, where the revival effect was considered in periodic two-component linear systems of dispersive equations and in periodic dispersive hyperbolic equations, in the case of other types of time-evolution problems with periodic or non-periodic boundary conditions, the examination of the revival effect has been considered only on a case-by-case scenario. For instance, we refer to \cite{olver2018revivals}, \cite{ boulton2020new}, \cite{boulton2021beyond}. For further context see the review paper \cite{smith2020revival}. 

The paper is based on Chapter 8 of the author's PhD thesis \cite{farmakis2022revival}, in which the revival effect was examined in the quasi-periodic problem \eqref{Quasi-Periodic Problem} only for the monomial case $P(\lambda) = \lambda^{n}$. The structure of the paper is as follows. In Section~\ref{Statement of main results} we present the main results and briefly describe our methodology. The full proofs are given in Section~\ref{Proof of main results}. In Section~\ref{Second-order case and fractal dimensions}, we make a few remarks on the second-order case and describe the consequences of the analysis on the fractal dimension of the solutions to higher-order models. Notice that in \eqref{Quasi-Periodic Problem}, we have excluded dispersion relations given by a real quadratic polynomial. This case has been addressed extensively in the literature, both in the periodic and the quasi-periodic setting. For example, it includes the free linear Schr\"{o}dinger equation ($P(\lambda) = \lambda^{2}$), which is regarded as the standard model for the study of revivals. Numerical examples that illustrate the validity of our results are included in Section~\ref{Numerical Examples}. In Section~\ref{Revivals in non-linear equations}, we apply our method to the cubic non-linear Schr\"{o}dinger equation, and thus we extend the revival effect  from the periodic setting, considered in \cite{erdogan2013talbotpaper}, to the quasi-periodic case. We also briefly describe two possible approaches in order to examine the revival effect in the Korteweg-de Vries equation with quasi-periodic boundary conditions. A summary of our findings is given in Section~\ref{Conclusion}.

\section{Statement of main results}\label{Statement of main results}

One of our main contributions is the following theorem which describes the influence of quasi-periodic boundary conditions on the behaviour of the solution at rational times.

\begin{theorem}
    \label{Main theorem}
    Fix $\theta\in (0,1)$, integer $n\geq3$  and let $p$ and $q$ be positive, co-prime integers. Assume that $u_{0}\in L^{2}(0,2\pi)$. Then, at rational times $t_{r} = 2\pi \frac{p}{q}$, the solution to the quasi-periodic problem \eqref{Quasi-Periodic Problem} satisfies the following.
    \begin{itemize}
        \item [(i)] If $\theta \in \mathbb{Q}$, then $u(x,t_{r})$ is a finite linear combination of translated copies of $u_{0}$.
        \item [(ii)] If $\theta\not\in\mathbb{Q}$ and $u_{0}$ is of bounded variation over $[0,2\pi]$,
        then $u(x,t_{r})$ is a continuous function of $x\in [0,2\pi]$ and $e^{i2\pi\theta}u(0,t_{r})=u(2\pi,t_{r})$.
    \end{itemize}
\end{theorem}

According to Theorem~\ref{Main theorem} the regularity of the solution at rational times depends on the value of the quasi-periodic parameter $\theta$. Indeed, given a piece-wise continuous initial function $u_{0}$ of bounded variation on $[0,2\pi]$, the solution at rational times displays the following dichotomy.
 
\begin{itemize}
    \item For $\theta \in\mathbb{Q}$, $u(x,t_{r})$ is piece-wise continuous. Indeed, it is given by a finite linear combination of translations of $u_{0}$. Hence, the revival effect persists in this case.
    \item For $\theta\not\in\mathbb{Q}$, any initial jump discontinuity disappears and the solution becomes a continuous function of $x\in [0,2\pi]$, satisfying the boundary condition $e^{i2\pi\theta}u(0,t_{r})=u(2\pi,t_{r})$. Hence, there is no revival of the initial jump discontinuities. 
\end{itemize}

Therefore, in contrast to the classical (periodic) revival result by Olver \cite{olver2010dispersive}, Theorem~\ref{Main theorem} implies that, in general, the quasi-periodic problem \eqref{Quasi-Periodic Problem} does not exhibit revivals at rational times. The revival property survives only when $\theta$ is rational.

For the proof of Theorem~\ref{Main theorem}, we develop a different approach than the one in \cite{boulton2021beyond} for the Airy PDE. The latter was based on the Fourier method in connection with the orthonormal basis of $L^{2}(0,2\pi)$
\begin{equation}
    \label{quasi-periodic basis}
    \phi_{j}(x) = \frac{e^{i(j+\theta)x}}{\sqrt{2\pi}}, \quad j\in\mathbb{Z}.
\end{equation}
Note that $\{\phi_{j}\}_{j\in\mathbb{Z}}$ is the family of eigenfunctions of the underlying spatial linear differential operator $P(-i\partial_{x})$ in \eqref{Quasi-Periodic Problem}, defined on an appropriate dense domain of $L^{2}(0,2\pi)$ incorporating the quasi-periodic boundary conditions. The associated eigenvalues are $\{P(j+\theta)\}_{j\in\mathbb{Z}}$. Thus, the existence and uniqueness of an $L^{2}(0,2\pi)$ solution to \eqref{Quasi-Periodic Problem} for any initial data $u_{0} \in L^{2}(0,2\pi)$ follows from the Fourier method.

Instead of analysing the revival effect via the generalised Fourier series representation of the solution, below we consider the revival effect in the following \emph{periodic problem} posed on $[0,2\pi]$,
\begin{equation}
    \label{Periodic problem}
        \begin{aligned}
        &\partial_{t}z(x,t) = -i A(-i\partial_{x}) z(x,t), \quad z(x,0) = z_{0}(x) \\
        & \partial_{x}^{m}z(0,t) = \partial_{x}^{m}z(2\pi,t), \quad m=0,1,\dots,n-1,
        \end{aligned}
    \end{equation}
where $A$ is the polynomial (of order $n\geq 3$) given by 
\begin{equation}
        \label{Polynomial A}
        A(\lambda) = \sum_{m=2}^{n} \alpha_{m}\sum_{k=2}^{m} \pmqty{m\\k} \theta^{m-k} \lambda^{k}.
\end{equation}
The periodic problem \eqref{Periodic problem} arises naturally by considering the transformation
\begin{equation}
    \label{Correspondence transformation inverse}
    u(x,t) = e^{-i(P(\theta) t  - \theta x)} z^{*}(x-s_{\theta} t,t)
\end{equation}
where, $z^{*}$ denotes the $2\pi$-periodic extension of $z$ in the space variable, see \eqref{periodic extension}, and $s_{\theta}$ is a real constant which depends on $\theta$ and is defined by \eqref{s theta constant}. 

As we show later in Lemma~\ref{Correspondece lemma}, the solution to the quasi-periodic problem \eqref{Quasi-Periodic Problem}, with initial condition $u_{0}(x)$, is given at any time $t\geq0$ via \eqref{Correspondence transformation inverse}, where $z(x,t)$ solves the periodic problem \eqref{Periodic problem} with initial condition $z_{0}(x) = e^{-i\theta x}u_{0}(x)$.

Consequently, it is enough to examine the revivals only in the periodic problem and then transfer our results to the quasi-periodic problem by means of \eqref{Correspondence transformation inverse}. The next theorem describes the behaviour of the solution to the periodic problem \eqref{Periodic problem} at rational times. It is a recasting of Theorem~\ref{Main theorem}.

\begin{theorem}
    \label{Main theorem 2}
    Fix $\theta\in (0,1)$, integer $n\geq3$  and let $p$ and $q$ be positive, co-prime integers. Assume that $z_{0}\in L^{2}(0,2\pi)$. Then, at rational times $t_{r} = 2\pi \frac{p}{q}$, the solution to the periodic problem \eqref{Periodic problem} satisfies the following.
    \begin{itemize}
        \item [(i)] If $\theta \in \mathbb{Q}$, then $z(x,t_{r})$ is a finite linear combination of periodically translated copies of $z_{0}$.
        \item [(ii)] If $\theta\not\in\mathbb{Q}$ and $z_{0}$ is of bounded variation over $[0,2\pi]$, then $z(x,t_{r})$ is a continuous function of $x\in [0,2\pi]$ and such that $z(0,t_{r})=z(2\pi,t_{r})$.
    \end{itemize}
\end{theorem}

Theorem~\ref{Main theorem 2} implies that whenever $\theta$ is not rational, the periodic problem \eqref{Periodic problem} does not support the revival effect at rational times. This immediatly yields the conclusion of Theorem~\ref{Main theorem} through \eqref{Correspondence transformation inverse}. More generally, we may deduce that linear dispesvive PDEs with dispersion relation a polynomial of order $n\geq 3$, exhibit revivals at rational times if and only if the coefficients of the polynomial are all rational (since after appropriate scaling, we can reduce to integer coefficients). Again, notice that Theorem~\ref{Main theorem 2} is in contrast to the classical revival theory \cite{olver2010dispersive} by identifying a class of periodic problems for linear dispersive PDEs of order $n\geq 3$, that does not in general support the revival effect at rational times. 

With respect to the revival property,  the connection between the two time-evolution problems, \eqref{Quasi-Periodic Problem} and \eqref{Periodic problem}, is now clear. Roughly, given the correspondence \eqref{Correspondence transformation inverse}, Theorems~\ref{Main theorem} and~\ref{Main theorem 2} are equivalent. The revival effect is present at rational times in the quasi-periodic problem \eqref{Quasi-Periodic Problem} if and only if it is present at rational times in the periodic problem \eqref{Periodic problem}, which holds if and only if $\theta$ is rational.

The proof of Theorem~\ref{Main theorem 2} will be obtained by deriving at any positive time a solution representation to the periodic problem \eqref{Periodic problem} in terms of an iteration of compositions of the solution operators to simpler periodic problems which exhibit the revival/fractalisation dichotomy at rational and irrational times. This is the context of the next section which gives the proof of our main two results, Theorem~\ref{Main theorem} and \ref{Main theorem 2}.  


\section{Proof of main results}\label{Proof of main results}

In the first part of this section we prove Lemma~\ref{Correspondece lemma} which implies \eqref{Correspondence transformation inverse}. The proof of Theorem~\ref{Main theorem 2} is given in the second part. Then, as described above, Theorem~\ref{Main theorem} follows immediately.  

First, let us establish the notation. Because the revival phenomenon is given in terms of translations, we require extending functions outside $[0,2\pi]$. More specifically, we consider the periodic extension of a function and its translation.
\begin{definition}
    Given a function $f$ on $[0,2\pi)$, we denote by $f^{*}$ the \emph{$2\pi$-periodic extension} to $\mathbb{R}$, which is defined as follows
\begin{equation}
    \label{periodic extension}
    f^{*}(x)  = f(x-2\pi m), \quad 2\pi m \leq x< 2\pi (m+1), \ m\in\mathbb{Z}. 
\end{equation}
\end{definition}

\begin{definition}
    For fixed $s\in \mathbb{R}$, we denote by $\mathcal{T}_{s}$ \emph{the periodic translation operator} on $L^{2}(0,2\pi)$, which is the defined as follows
\begin{equation}
    \label{Periodic Translation Operator}
    \mathcal{T}_{s} f (x) = f^{*}(x-s), \quad f \in L^{2}(0,2\pi). 
\end{equation}
\end{definition}

We also introduce the following real constant which depends on $\theta$.  
\begin{definition}
    For fixed $\theta\in (0,1)$, we define the real constant $s_{\theta}$ by
    \begin{equation}
        \label{s theta constant}
            s_{\theta} = P'(\theta) = \sum_{m=1}^{n} m \alpha_{m} \theta^{m-1} \in\mathbb{R},
    \end{equation}
where $P'$ is the derivative of the polynomial $P$ given by \eqref{Polynomial P}.
\end{definition}
The translation by $s_{\theta} t$ enables the connection between the periodic \eqref{Periodic problem} and the quasi-periodic \eqref{Quasi-Periodic Problem} problems, see \eqref{Correspondence transformation inverse} or equivalently \eqref{Correspondece transformation}.

\subsection{Periodic and quasi-periodic correspondences}

The underline key idea in the transformation \eqref{Correspondence transformation inverse} originates from the (invertible) map
\begin{equation}
	\label{Quasi-periodic Transformation}
	g(x) = e^{-i\theta x} f(x),
\end{equation}
which implies that if $e^{i2\pi\theta} f(0) = f(2\pi)$, then $g$ is $2\pi$-periodic. We can view \eqref{Quasi-periodic Transformation}, as a transformation of quasi-periodic to periodic boundary conditions, and vice-versa, similar to the spectral analysis encountered in Floquet's theory of the linear Schr\"{o}dinger equation with periodic potential \cite[Section 7.6]{borthwick2020spectral}.

In the next lemma, we first consider a modification of \eqref{Quasi-periodic Transformation}, that incorporates the time dependence and give a first correspondence between the quasi-periodic problem \eqref{Quasi-Periodic Problem} and a periodic problem. Recall that $P$, $A$ and $s_{\theta}$ are given by \eqref{Polynomial P}, \eqref{Polynomial A} and \eqref{s theta constant} respectively.

\begin{lemma}
    \label{Correspondece lemma first}
    Fix $\theta \in (0,1)$ and an integer $n\geq 3$. Consider the transformation
    \begin{equation}
    \label{Correspondece transformation first}
        w(x,t) = e^{i (P(\theta) t - \theta x)} u(x,t).
    \end{equation}
    Then, the function $w=w(x,t)$ satisfies the periodic time-evolution problem
    \begin{equation}
		\label{First Periodic Problem}
		\begin{aligned}
			&\big(\partial_{t}+ s_{\theta} \partial_{x}\big)w(x,t) = -i A(-i\partial_{x}) w(x,t), \quad w(x,0) = e^{-i \theta x}u_{0}(x), \\ 
			&\partial_{x}^{m}w(0,t) = \partial_{x}^{m}w(2\pi,t),\quad m=0,1,\dots,n-1,
		\end{aligned}
	\end{equation}
    if and only if the function $u=u(x,t)$ satisfies the quasi-periodic problem \eqref{Quasi-Periodic Problem} with $u(x,0)=u_{0}(x)$. 
\end{lemma}
\begin{proof}
We assume that the function $u=u(x,t)$ is the solution to the quasi-periodic problem \eqref{Quasi-Periodic Problem}. We show that if $w=w(x,t)$ is given by \eqref{Correspondece transformation first}, then it satisfies the time-evolution problem \eqref{First Periodic Problem}. The converse implication is also true by reversing the argument.

Taking the time derivative in \eqref{Correspondece transformation first} and using that $\partial_{t}u = -iP(-i\partial_{x}) u$, we have that
	\begin{equation}
		\label{FTP1}
		\begin{aligned}
			\partial_{t} w(x,t) &= i P(\theta) w(x,t) + \sum_{m=0}^{n} (-i)^{m+1} \alpha_{m} e^{i( P(\theta) t - \theta x)} \partial_{x}^{m}u(x,t).
		\end{aligned}
	\end{equation}
We claim that
	\begin{equation}
		\label{FTP2}
		e^{i(P(\theta) t - \theta x)} \partial_{x}^{m}u(x,t) = \big(\partial_{x} +i \theta\big)^{m} w(x,t), \quad \forall \ m\in\mathbb{N}.
	\end{equation}
	It is straightforward to see that the claim holds for $m=1$ and $2$, and we proceed by induction. Given that it is true for $m$, we examine the case $m+1$. We have
	\begin{equation*}
		\partial_{x}\Big(e^{i(P(\theta)t - \theta x)}\partial_{x}^{m}u(x,t)\Big) = \partial_{x}\Big(\big(\partial_{x} +i \theta\big)^{m} w(x,t)\Big)
	\end{equation*}
	which gives
	\begin{equation*}
		e^{i(P(\theta)t - \theta x)}\partial_{x}^{m+1}u(x,t) = \Big(\partial_{x}\big(\partial_{x} +i\theta\big)^{m} + i\theta\big(\partial_{x} + i\theta\big)^{m} \Big)w(x,t).
	\end{equation*}
	Using the Binomial Theorem, we have that
	\begin{equation*}
		\begin{aligned}
			e^{i(P(\theta)t - \theta x)}\partial_{x}^{m+1}u(x,t) &= \Bigg(\partial_{x}\Big(\sum_{k=0}^{m}\pmqty{m\\k} (i\theta)^{m-k}\partial_{x}^{k}\Big) + i\theta\Big(\sum_{k=0}^{m}\pmqty{m\\k}(i\theta)^{m-k}\partial_{x}^{k}\Big)\Bigg) w(x,t) \\ \\
			&= \Bigg(\sum_{k=0}^{m}\pmqty{m\\k}(i\theta)^{m-k}\partial_{x}^{k+1} + \sum_{k=0}^{m}\pmqty{m\\k} (i\pi\theta)^{m-k+1}\partial_{x}^{k}\Bigg)w(x,t).
			\end{aligned}
	\end{equation*}
We expand the sums over $k$ and find that 
	\begin{equation*}
		\begin{aligned}
			e^{i(P(\theta)t - \theta x)}\partial_{x}^{m+1}u(x,t)
			& = \Bigg(\partial_{x}^{m+1} + \pmqty{m\\m-1} i\theta \partial_{x}^{m} + \pmqty{m\\m-2} (i \theta)^{2} \partial_{x}^{m-1} +\dots \\ 
			&\qquad + \pmqty{m\\0} (i\theta)^{m} \partial_{x} + \pmqty{m\\m} i\theta \partial_{x}^{m} + \pmqty{m\\m-1} (i\theta)^{2} \partial_{x}^{m-1} + \dots \\
			&\qquad +\pmqty{m\\1} (i\theta)^{m} \partial_{x} + (i\theta)^{m+1} \Bigg) w(x,t) \\
			& = \Bigg(\partial_{x}^{m+1} + \pmqty{m+1\\m} i\theta \partial_{x}^{m}  + \pmqty{m+1\\m-1} (i\theta)^{2} \partial_{x}^{m-1}  + \dots  \\
			& \qquad + \pmqty{m+1\\1} (i\theta)^{m} \partial_{x}  + (i\theta)^{m+1}\Bigg)w(x,t)  \\
			& = \Bigg(\sum_{k=0}^{m+1}\pmqty{m+1\\k}\big(i\theta)^{m+1 - k}\partial_{x}^{k}\Bigg)w(x,t) \\ 
			& = \big(\partial_{x} + i\theta \big)^{m+1} w(x,t).
		\end{aligned}
	\end{equation*}
	Hence, \eqref{FTP2} holds true and by substitution into \eqref{FTP1}, we arrive at
	\begin{equation*}
		\begin{aligned}
			\partial_{t}w(x,t) &= iP(\theta) w(x,t) + \sum_{m=0}^{n}(-i)^{m+1}\alpha_{m}\big(\partial_{x} + i\theta\big)^{m} w(x,t) \\ 
			&= iP(\theta) w(x,t) + \sum_{m=0}^{n}(-i)^{m+1}\alpha_{m} \sum_{k=0}^{m}\pmqty{m\\k} (i\theta)^{m-k}\partial_{x}^{k}w(x,t) \\ 
			&=iP(\theta) w(x,t) -i \Big(\alpha_{0} + \alpha_{1}\theta + \sum_{m=2}^{n} (-i)^{m} \alpha_{m}(i\theta)^{m}\Big) w(x,t) \\
            & \quad  - \Big(\alpha_{1} - \sum_{m=2}^{n} (-i)^{m+1} \alpha_{m} m (i\theta)^{m-1}\Big)\partial_{x} w(x,t) \\
            & \quad + \sum_{m=2}^{n} (-i)^{m+1} \alpha_{m}\sum_{k=2}^{m}\pmqty{m\\k}(i\theta)^{m-k} \partial_{x}^{k}w(x,t).
		\end{aligned}
	\end{equation*}
    Thus,
	\begin{equation*}
		\big(\partial_{t} + s_{\theta} \partial_{x}\big)w(x,t) = -i A(-i \partial_{x}) w(x,t),
	\end{equation*}
	as claimed. Moreover, at time $t=0$, we obtain from \eqref{Correspondece transformation first}, the initial condition $w(x,0) = e^{-i \theta x} u_{0}(x)$.
	
    Finally, we need to show that $w(x,t)$ satisfies the periodic boundary conditions on the interval $[0,2\pi]$. For any $m=0,1,\dots,n-1$ we have from \eqref{FTP2} at $x=2\pi$,
	\begin{equation*}
		\begin{aligned}
			\partial_{x}^{m}w(2\pi,t) &=  e^{iP(\theta) t} e^{-i2\pi \theta} \partial_{x}^{m}u(2\pi,t) - \sum_{k=0}^{m-1}\pmqty{m\\k} (i\theta)^{m-k}\partial_{x}^{k}w(2\pi,t) \\
			& =  e^{i P(\theta) t} \partial_{x}^{m}u(0,t) - \sum_{k=0}^{m-1}\pmqty{m\\k}(i \theta)^{m-k} \partial_{x}^{k}w(2\pi,t).
		\end{aligned}
	\end{equation*}
	As $w(0,t) = w(2\pi,t)$ and $\partial_{x}w(2\pi,t) = \partial_{x}w(2\pi,t)$, assuming that $\partial_{x}^{k}w(0,t) = \partial_{x}^{k}w(2\pi,t)$ holds for $k=0,1,\dots,m-1$, the above equation gives 
	\begin{equation*}
		\partial_{x}^{m}w(2\pi,t) = e^{i P(\theta) t}\partial_{x}^{m}u(0,t) - \sum_{k=0}^{m-1}\pmqty{m\\k} (i\theta)^{m-k} \partial_{x}^{k}w(0,t) = \partial_{x}^{m}w(0,t).
	\end{equation*}
\end{proof}

The next lemma gives the central correspondence between the quasi-periodic \eqref{Quasi-Periodic Problem} and periodic \eqref{Periodic problem} problems. The solution to the periodic problem is given via the periodic translation $\mathcal{T}_{-s_{\theta t}}$ of the solution to the quasi-periodic problem, essentially via the application of a Galilean transformation, after multiplying by  $e^{i (P(\theta) t - \theta x)}$. Since the periodic translation operator is unitary, we can apply the inverse $\mathcal{T}_{s_{\theta}t}$ and obtain the solution of the quasi-periodic problem in terms of the solution to the periodic problem, see \eqref{Correspondence transformation inverse}
or \eqref{Correspondence transformation inverse 1} below. Recall that $P$, $\mathcal{T}_{-s_{\theta}t}$ and $s_{\theta}$ are given by \eqref{Polynomial P}, \eqref{Periodic Translation Operator} and \eqref{s theta constant} respectively.
 
\begin{lemma}
    \label{Correspondece lemma}
    Fix $\theta \in (0,1)$ and an integer $n\geq 3$. Consider the transformation
    \begin{equation}
    \label{Correspondece transformation}
        z(x,t) = e^{i ((P(\theta) - \theta s_{\theta})t-\theta x)} \mathcal{T}_{-s_{\theta} t} u(x,t).
    \end{equation}
     Then, the function $z=z(x,t)$ satisfies the periodic problem \eqref{Periodic problem} with initial data $z_{0}(x) = e^{-i\theta x} u_{0}(x)$ if and only if the function $u=u(x,t)$ satisfies the quasi-periodic problem \eqref{Quasi-Periodic Problem} with initial data $u_{0}(x)$.
\end{lemma}
\begin{proof}
From Lemma~\ref{Correspondece lemma first}, we know that if $u(x,t)$ satisfies the quasi-periodic problem \eqref{Quasi-Periodic Problem}, then the function
	\begin{equation}
		w(x,t) = e^{i ( P(\theta)t - \theta x)} u(x,t)
	\end{equation}
	satisfies the periodic problem \eqref{First Periodic Problem}. We apply on $w(x,t)$ the Galilean transformation 
	\begin{equation*}
		x = y + s_{\theta} \tau, \quad t=\tau.
	\end{equation*} 
	Notice that,
	\begin{equation*}
		\partial_{\tau}w = \big(\partial_{t} + s_{\theta} \partial_{x}\big)w, \quad \partial_{y}^{m}w = \partial_{x}^{m}w, \quad m=1,\dots,n.
	\end{equation*}
	Therefore, the function
	\begin{equation*}
		z(y,\tau) = w(y+ s_{\theta} \tau,\tau) = e^{-i \big( ( P(\theta) - \theta s_{\theta})\tau  -\theta y\big)} \mathcal{T}_{- s_{\theta}\tau} u(y,\tau)
	\end{equation*}
	satisfies the periodic problem \eqref{Periodic problem}. The converse implication follows similarly by applying on $z(x,t)$ the transformation  $x = y - s_{\theta}\tau$, $t=\tau$ and then Lemma~\ref{Correspondece lemma first}.
\end{proof}

Solving for $u(x,t)$ in \eqref{Correspondece transformation} we have that 
\begin{equation}
    \label{Correspondence transformation inverse 1}
    u(x,t) = e^{-i(P(\theta) t  - \theta x)} \mathcal{T}_{s_{\theta} t} z(x,t),
\end{equation}
which express the relation \eqref{Correspondence transformation inverse}, but in terms of the periodic translation operator. So, by Lemma~\ref{Correspondece lemma}, the solution $u(x,t)$ of the quasi-periodic problem \eqref{Quasi-Periodic Problem} is given by \eqref{Correspondence transformation inverse 1} in terms of the solution $z(x,t)$ of the periodic problem \eqref{Periodic problem}, with $z_{0}(x) = e^{-i \theta x} u_{0}(x)$.

From now on, we focus only on the periodic problem for the formulation of the revival effect. Indeed at any time $t$, and thus also at rational times, the behaviour of the quasi-periodic problem is determined by the behaviour of the periodic problem via \eqref{Correspondence transformation inverse 1}. 

\subsection{Analysis of the revival effect}

In this part, we prove Theorem~\ref{Main theorem 2} and describe the behaviour of the solution to the periodic problem \eqref{Periodic problem} at rational times ($t/2\pi\in\mathbb{Q}$). Then, as a consequence of Theorem~\ref{Main theorem 2} and in conjunction with \eqref{Correspondence transformation inverse 1}, Theorem~\ref{Main theorem} follows immediately.

Theorem~\ref{Main theorem 2} is derived as a corollary of a special representation of the solution at any fixed time $t\geq 0$, see Proposition~\ref{PPP Solution Representation Proposition}. The representation is given by a composition of solution operators associated with the family of purely periodic problems on $[0,2\pi]$
\begin{equation}
	\label{Purely Periodic Problems}
	\begin{aligned}
		&\partial_{t}v(x,t) = -i (-i\partial_{x})^{n}v(x,t),\quad v(x,0) = v_{0}(x),\\
		&\partial_{x}^{m}v(0,t) = \partial_{x}^{m}v(a,t),\quad m=0,1,\dots,n-1,
	\end{aligned}
\end{equation}
where $n\in\mathbb{N}$. The consideration of \eqref{Purely Periodic Problems} allows a simple and clear application of known results on the revival and fractalisation phenomena as we precisely illustrate below.

It is known that for any initial condition $v_{0}$ in $L^{2}(0,2\pi)$, there exists in $L^{2}(0,2\pi)$ a unique solution (in the weak sense, see \cite{ball1977strongly}) to \eqref{Purely Periodic Problems}. In particular at any fixed time $t\geq0$, the solution has the Fourier series representation given by
\begin{equation*}
    v(x,t) =  \sum_{j\in\mathbb{Z}} \widehat{v_{0}}(j)e^{-ij^{n} t} e_{j}(x),
\end{equation*}
where 
\begin{equation}
    \label{Fourier basis}
    e_{j}(x) = \frac{e^{ijx}}{\sqrt{2\pi}}, \quad \widehat{v_{0}}(j) = \langle  v_{0}, e_{j} \rangle 
    = \int_{0}^{2\pi} v_{0}(x) \overline{e_{j}(x)} dx,
\end{equation}
with $\langle \cdot, \cdot \rangle$ denoting the standard inner-product in $L^{2}(0,2\pi)$.

In order to simplify our arguments below, we introduce the solution operators corresponding to the family of time-evolution problems \eqref{Purely Periodic Problems}. 
\begin{definition}
    For each $n\in\mathbb{N}$ and fixed $t\in\mathbb{R}$, we define the map $\mathcal{R}_{n}(t):L^{2}(0,2\pi)\rightarrow L^{2}(0,2\pi)$ by
\begin{equation}
	\label{Periodic group}
	 \quad \mathcal{R}_{n}(t)v_{0} (x)= \sum_{j\in\mathbb{Z}} \widehat{v_{0}}(j)e^{-ij^{n} t} e_{j}(x).
\end{equation}
\end{definition}

In the special case of $n = 1$, $\mathcal{R}_{1}(t)$ coincides with the definition of the periodic translation operator $\mathcal{T}_{t}$, since $\langle T_{t}v_{0}, e_{j}\rangle = \widehat{v_{0}}(j) e^{-ijt}$. In general, for each $n\in\mathbb{N}$, the family $\{\mathcal{R}_{n}(t)\}_{t\in\mathbb{R}}$ defines a group of strongly-continuous unitary operators in $L^{2}(0,2\pi)$, parameterised by $t\in\mathbb{R}$. So, for a given $v_{0}\in L^{2}(0,2\pi)$, we can write the solution to \eqref{Purely Periodic Problems} as 
\begin{equation}
    \label{Solution Purely Periodic Problems}
    v(x,t) = \mathcal{R}_{n}(t) v_{0}(x),
\end{equation} 
at time $t\geq 0$. 

In the next crucial statement we formulate  an effective representation of the solution $z(x,t)$
to the periodic problem \eqref{Periodic problem}, which enables to isolate the behaviour at rational times. The representation is given in terms of compositions of the operators $\mathcal{R}_{n}(t)$, for $n\in\mathbb{N}$ and $t\in\mathbb{R}$.

\begin{proposition}
	\label{PPP Solution Representation Proposition}
	Fix $\theta\in(0,1)$ and an integer $n\geq 3$. Then, for any $z_{0}\in L^{2}(0,2\pi)$, the solution to \eqref{Periodic problem} at each fixed time $t\geq 0$ admits the representation
    \begin{equation*}
        \label{PPP Solution Representation}
            z(x,t) = \prod_{m=2}^{n-1} \prod_{k=2}^{m} \mathcal{R}_{k} \Bigg(\alpha_{m+1}\pmqty{m+1\\k} \theta^{m+1-k} t\Bigg) \Big[\prod_{\ell = 2}^{n} \mathcal{R}_{\ell} (\alpha_{\ell} t)z_{0}(x)\Big].
        \end{equation*}
\end{proposition}
\begin{proof}
	For any $t\geq 0$, the solution in $L^{2}(0,2\pi)$ to \eqref{Periodic problem} is given by 
    \begin{equation*}
        z(x,t) = \sum_{j\in\mathbb{Z}} \widehat{z_{0}} e^{-iA(j) t} e_{j}(x).
    \end{equation*}
We can write $e^{-iA(j)t}$ as follows
    \begin{equation*}
    \begin{aligned}
		e^{-i A(j) t} &= \prod_{m=2}^{n}\prod_{k=2}^{m} \exp(-i \alpha_{m} j^{k} \pmqty{m\\k} \theta^{m-k} t) \\
        &= \prod_{m=2}^{n-1}\prod_{k=2}^{m} \exp(-i \alpha_{m+1} j^{k} \pmqty{m+1\\k} \theta^{m+1-k} t) \prod_{\ell=2}^{n} e^{-i\alpha_{\ell}j^{\ell} t}.
    \end{aligned}
	\end{equation*}
Therefore, by using \eqref{Periodic group}, we have that
	\begin{equation*}
		\begin{aligned}
			z(x,t) &= \sum_{j\in\mathbb{Z}} \widehat{z}_{0}(j) \prod_{m=2}^{n-1}\prod_{k=2}^{m} \exp(-i \alpha_{m+1} j^{k} \pmqty{m+1\\k} \theta^{m+1-k} t) \prod_{\ell=2}^{n} e^{-i\alpha_{\ell}j^{\ell} t} e_{j}(x) \\
			&=\prod_{m=2}^{n-1} \prod_{k=2}^{m} \mathcal{R}_{k} \Bigg(\alpha_{m+1}\pmqty{m+1\\k} \theta^{m+1-k} t\Bigg)\Big[  \sum_{j\in\mathbb{Z}} \widehat{z}_{0}(j) \prod_{\ell=2}^{n} e^{-i\alpha_{\ell}j^{\ell} t} e_{j}(x)\Big] \\
            &= \prod_{m=2}^{n-1} \prod_{k=2}^{m} \mathcal{R}_{k} \Bigg(\alpha_{m+1}\pmqty{m+1\\k} \theta^{m+1-k} t\Bigg)\Big[ \prod_{\ell=2}^{n} \mathcal{R}_{\ell}(\alpha_{\ell} t) z_{0}(x)\Big].
		\end{aligned}
	\end{equation*}
 This is exactly the representation in the statement.
\end{proof}

The representation in Proposition~\ref{PPP Solution Representation Proposition} allows us to explore the revival effect at rational times for the periodic problem \eqref{Periodic problem} and prove Theorem~\ref{Main theorem 2} by applying the following known result for the simpler periodic problems \eqref{Purely Periodic Problems}. A full proof of this statement is omitted and can be found in \cite[Theorems 2.14 and 2.16]{erdougan2016dispersive}. 

\begin{theorem}[\cite{erdougan2016dispersive}]  
    \label{Periodic Revival Theorem}
      Fix integer $n\geq 2$ and $v_{0}\in L^{2}(0,2\pi)$. Consider at time $t>0$, the solution $v(x,t) = \mathcal{R}_{n}(t) v_{0}(x)$ to the time-evolution problem \eqref{Purely Periodic Problems}. Then, we have the following.  
    \begin{itemize}
        \item [(i)] If $t = 2\pi\frac{p}{q}$, where $p$ and $q$ are positive, co-prime integers, then 
        \begin{equation}
        \label{Periodic Revival}
        \mathcal{R}_{n}\Big(2\pi\frac{p}{q}\Big) v_{0}(x) = \frac{1}{q}\sum_{k,m=0}^{q-1} e^{2\pi i ( - m^{k} \frac{p}{q} + m \frac{n}{q})} \mathcal{T}_{2\pi\frac{k}{q}} v_{0}(x). 
    \end{equation}
        \item [(ii)] If $t/2\pi\not\in\mathbb{Q}$ and $v_{0}$ is of bounded variation over $[0,2\pi]$, then $v(x,t) = \mathcal{R}_{n}(t)v_{0}(x)$ is a continuous function of $x\in [0,2\pi]$ and such that $v(0,t)=v(2\pi,t)$.
    \end{itemize}
\end{theorem}

Due to Theorem~\ref{Periodic Revival Theorem}, we notice that at rational times $t = 2\pi \frac{p}{q}$, the solution operator $\mathcal{R}_{n}(t)$ is a finite linear combination of periodic translation operators 
    \begin{equation}
    \label{Revival operators}
    \mathcal{R}_{n}\big(2\pi\frac{p}{q}\big) = \frac{1}{q}\sum_{k,m=0}^{q-1} e^{2\pi i ( - m^{n} \frac{p}{q} + m \frac{k}{q})} \mathcal{T}_{2\pi\frac{k}{q}}. 
\end{equation}
Following \cite{boulton2021beyond}, we refer to $\mathcal{R}_{n}(2\pi p/q)$ as the \emph{periodic revival operator} of order $n$ at $(p,q)$. The revival operators are a tool to identify the revival effect and at the same time they provide a compact notation for it. 

We are now in place to establish Theorem~\ref{Main theorem 2}. It is a consequence of Proposition~\ref{PPP Solution Representation Proposition} and Theorem~\ref{Periodic Revival Theorem}. The next corollary gives part $(i)$ of the theorem.

\begin{corollary} [Theorem~\ref{Main theorem 2}-(i)]
    \label{Corollary 2}
    Fix integer $n\geq 3$, $\theta\in\mathbb{Q}\cap (0,1)$ and consider the periodic problem \eqref{Periodic problem} for $z=z(x,t)$ with initial condition $z_{0}\in L^{2}(0,2\pi)$. Then, at any rational time $t_{r} = 2\pi \frac{p}{q}$, the solution is given by a finite linear combination of periodic translations of the initial data $z_{0}$.
\end{corollary}
\begin{proof}
Let $\theta = \frac{p_{1}}{q_{1}}\in\mathbb{Q}$. Then, from Proposition~\ref{PPP Solution Representation Proposition} we have that
\begin{equation*}
    z(x,t_{r}) = \prod_{m=2}^{n-1} \prod_{k=2}^{m} \mathcal{R}_{k} \Bigg(\alpha_{m+1}\pmqty{m+1\\k} \frac{p_{1}^{m+1-k}}{q_{1}^{m+1-k}} t_{r}\Bigg)\Big[ \prod_{\ell=2}^{n} \mathcal{R}_{\ell}(\alpha_{\ell} t_{r}) z_{0}(x)\Big].
\end{equation*}
Notice that because for each $\ell = 2, \dots, n$, $\alpha_{\ell}$ is an integer, it follows that the times $\alpha_{\ell} t_{r}$ are all rational times. Consequently, $\mathcal{R}_{\ell}(\alpha_{\ell} t_{r})$ represents a revival operator in accordance with \eqref{Revival operators}. Similarly, the times
\begin{equation*}
    \alpha_{m+1}\pmqty{m+1\\k} \frac{p_{1}^{m+1-k}}{q_{1}^{m+1-k}} t_{r} = 2\pi \frac{\alpha_{m+1} (m+1)! p_{1}^{m+1-k} p}{(m+1-k)! k! q_{1}^{m+1-k} q}
\end{equation*}
are also rational times since $m$, $k$, $p_{1}$, $q_{1}$, $p$, $q$ assume only integer values. Therefore, for each $m=2, \dots, (n-1)$ and $k=2,\dots,m$, the operators
\begin{equation*}
    \mathcal{R}_{k} \Bigg(\alpha_{m+1}\pmqty{m+1\\k} \frac{p_{1}^{m+1-k}}{q_{1}^{m+1-k}} t_{r}\Bigg)
\end{equation*}
represent revival operators given by \eqref{Revival operators}. 
\end{proof}

We now consider the case when $\theta \in (0,1)$, $\theta\not\in\mathbb{Q}$. The following corollary corresponds to part $(ii)$ of Theorem~\ref{Main theorem 2}.

\begin{corollary}[Theorem~\ref{Main theorem 2}-(ii)]
    \label{Corollary 3}
    Fix integer $n\geq 3$, $\theta\in (0,1)\setminus \mathbb{Q}$ and consider the periodic problem \eqref{Periodic problem} for $z=z(x,t)$. Assume that the initial condition $z_{0}$ is of bounded variation over $[0,2\pi]$. Then, at any rational time $t_{r}=2\pi \frac{p}{q}$, the solution is a continuous function of $x\in [0,2\pi]$ and such that $z(0,t) = z(2\pi,t)$.
\end{corollary}
\begin{proof}
    From Proposition~\ref{PPP Solution Representation Proposition}, we have that
\begin{equation*}
    z(x,t_{r}) = \prod_{m=2}^{n-1} \prod_{k=2}^{m} \mathcal{R}_{k} \Bigg(\alpha_{m+1}\pmqty{m+1\\k} \theta^{m+1-k} t_{r}\Bigg)\Big[ \prod_{\ell=2}^{n} \mathcal{R}_{\ell}(\alpha_{\ell} t_{r}) z_{0}(x)\Big].
\end{equation*}
First, notice that for each $\ell = 2, \dots, n$, $\alpha_{\ell}$ is an integer. Thus, it follows that the times $\alpha_{\ell} t_{r}$ are rational times. Consequently, due to \eqref{Revival operators}, $\mathcal{R}_{\ell}(\alpha_{\ell} t_{r} z_{0}(x))$ is a finite linear combination of periodic translations of $z_{0}$. For simplicity, set
\begin{equation*}
    v_{0}(x) = \prod_{\ell=2}^{n} \mathcal{R}_{\ell}(\alpha_{\ell} t_{r} z_{0}(x))
\end{equation*}
From the above we know that $v_{0}$ is of bounded variation over $[0,2\pi]$. On the other hand, because the parameter $\theta$ is a fixed irrational number in $(0,1)$ the times 
\begin{equation*}
    t_{\theta} =  2\pi \alpha_{m+1}\pmqty{m+1\\k} \theta^{m+1-k} \frac{p}{q}
\end{equation*}
are irrational times, and thus due to Theorem~\ref{Periodic Revival Theorem}-$(ii)$, for any $k=2,\dots, m$, with $m=2, \dots, n-1$, the functions $\mathcal{R}_{k} (t_{\theta}) v_{0}(x)$ are $2\pi$-periodic, continuous functions of $x\in [0,2\pi]$. 
\end{proof}

Finally, Theorem~\ref{Main theorem} comes as a result of Lemma~\ref{Correspondece lemma}, in particular from \eqref{Correspondence transformation inverse 1}, combined with Corollaries \ref{Corollary 2} and \ref{Corollary 3} or equivalently Theorem~\ref{Main theorem 2}.

\section{Second-order case and fractal dimensions}\label{Second-order case and fractal dimensions}

In the previous sections we did not consider second-order quasi-periodic problems 
\begin{equation}
\label{Second order Quasi-Periodic Problem}
    \begin{aligned}
    &\partial_{t}u(x,t) = -i \Big(\sum_{m=0}^{2} \alpha_{m} \big(-i\partial_{x}\big)^{m}\Big) u(x,t),\quad u(x,0) = u_{0}(x), \\
    &e^{i2\pi\theta} \partial_{x}^{m} u(0,t) = \partial_{x}^{m} u(2\pi,t), \quad m =0,1,
    \end{aligned}
\end{equation}
where $\alpha_{m}\in\mathbb{Z}$, $\alpha_{2}\not=0$. We now show that this case is isolated, as the revival effect is present at rational times not only for rational values of $\theta$ but for any $\theta\in (0,1)$. Indeed, first, note that Lemma~\ref{Correspondece lemma} holds for the case $n=2$, that is for the quasi-periodic problem \eqref{Second order Quasi-Periodic Problem}. It follows that the solution to \eqref{Second order Quasi-Periodic Problem} is given by
\begin{equation}
\label{Second order Correspondence}
    u(x,t) = \exp(-i( t \sum_{m=0}^{2}\alpha_{m}\theta^{m} - \theta x)) \mathcal{T}_{s_{\theta} t} z(x,t),
\end{equation}
where $s_{\theta} = \alpha_{1} + \alpha_{2}\theta \in \mathbb{R}$ and $z(x,t)$ solves the second-order periodic problem
\begin{equation}
\label{Second order Periodic Problem}
    \begin{aligned}
        &\partial_{t}z(x,t) = i\alpha_{2}\partial_{x}^{2}z(x,t), \quad z=(x,0)= z_{0}(x), \\
        &\partial_{x}^{m}z(0,t) = \partial_{x}^{m}z(2\pi,t), \quad m=0,1,
    \end{aligned}
\end{equation}
with $z_{0} (x) = e^{-i\theta x} u_{0}(x)$.

At any time $t\geq 0$, the solution to \eqref{Second order Periodic Problem} admits the Fourier series representation
\begin{equation*}
    z(x,t) = \sum_{j\in\mathbb{Z}} \widehat{z_{0}}(j) e^{-i\alpha_{2} j^{2} t} e_{j}(x),
\end{equation*}
or equivalently, $z(x,t) = \mathcal{R}_{2}(\alpha_{2} t)z_{0}(x)$, in terms of the unitary group $\mathcal{R}_{2}(t)$ defined by \eqref{Periodic group}. So, since $\alpha_{2}$ is an integer, we have that at rational times $t_{r}=2\pi\frac{p}{q}$, $\mathcal{R}_{2}(\alpha_{2}t_{r})$ defines a periodic revival operator of order 2 defined by \eqref{Revival operators}. The latter implies the validity of the revival effect in the periodic problem \eqref{Second order Periodic Problem} and consequently, via \eqref{Second order Correspondence}, in the quasi-periodic problem for any $\theta\in (0,1)$. In particular, notice that the transformation \eqref{Second order Correspondence} only changes the boundary conditions from quasi-periodic to periodic. It does not affect the structure of the PDE in \eqref{Second order Periodic Problem}, in the sense that the transformed equation also has integer coefficients which allows the revival effect to persist. The situation is even more clear in the case of the linear Schr\"{o}dinger equation. 

Indeed, for the values $a_{0} = a_{1} = 0$, $a_{2} = 1$, the second-order model \eqref{Second order Quasi-Periodic Problem} reduces to the free linear Schr\"{o}dinger equation (FLS) $\partial_{t} u  = i\partial_{x}^{2}u$. Observe that in this case the transformation \eqref{Second order Correspondence} leaves the FLS equation invariant, which is know as the Galilean invariance of the FLS equation \cite{linares2014introduction}, and only changes the boundary conditions from quasi-periodic in \eqref{Second order Quasi-Periodic Problem} to periodic in \eqref{Second order Periodic Problem}. Hence, the revival phenomenon in the quasi-periodic FLS equation follows immediately from the known periodic case and as a direct consequence of the Galilean invariance. This observation adds an alternative proof of the revival property in the quasi-periodic FLS equation, see \cite{olver2018revivals} and \cite{boulton2021beyond} respectively for two other proofs, where a more general class of quasi-periodic boundary conditions is treated. Moreover, as we briefly outline in Section~\ref{Revivals in non-linear equations}, the proof given here allows the consideration of the revivals in the cubic non-linear Schr\"{o}dinger equation with quasi-periodic boundary conditions. 

In the literature, the revival and fractalisation dichotomy in the FLS equation is also known as the \emph{Talbot effect}, a diffraction phenomenon in optics, see for instance \cite[Section 3.19]{gbur2011mathematical}. The connection between the Talbot effect and the revival and fractalisation phenomena was originated in the work of Berry and Klein~\cite{berry1996integer} in 1996, with the terminology being interchangeable nowadays. However, in 1992, Oskolkov \cite{oskolkov1992class} had already analysed to some extend the behaviour of the solution at rational and irrational times for bounded variation initial conditions. Other notable works include the paper of Taylor \cite{taylor2003schrodinger} in 2003, who re-discovered that the periodic FLS equation exhibits revivals at rational times and extended it to the higher-dimensional torus and sphere and the already mentioned work of Olver \cite{olver2010dispersive} in 2010, who named the revival effect \emph{dispersive quantisation} due to the dispersive nature of the equations and the property of their solution to quantised into a finite number of copies of the initial datum. Precise statements of some of these results can be found in \cite[Section 2.3]{erdougan2016dispersive}, as we have mentioned in previous parts of the paper.     

Furthermore, with respect to the fractalisation effect, Berry and Klein \cite{berry1996integer} conjectured that at irrational times the solution profiles of the periodic FLS equation have \emph{fractal dimension} $3/2$ (the term refers to the upper Minkowski dimension, otherwise known as box counting dimension, see \cite{mattila1999geometry}), for given piece-wise constant initial conditions. A full proof of this conjecture was given in 2000 by Rodnianski in \cite{rodnianski2000fractal} based on previous work with Kapitanski \cite{kapitanski1999does} for a general class of initial data with low Sobolev regularity.  

In order to draw some conclusions on the fractal dimension of the the real and imaginary parts of the solutions to quasi-periodic problems, below we present Rodnianski's result and its extension to the periodic problems \eqref{Purely Periodic Problems}, for $n\geq 3$, by Chousionis, Erdo\u{g}an and Tzirakis \cite[Theorem 3.12]{chousionis2014fractal}, see also \cite[Theorem 2.16]{erdougan2016dispersive}.  

\begin{theorem}[\cite{rodnianski2000fractal}, \cite{chousionis2014fractal,erdougan2016dispersive}]
\label{Periodic fractalisation}
Fix integer $n\geq 2$ and consider the periodic boundary value problem \eqref{Purely Periodic Problems}. 
Let $v_{0}$ be a function of bounded variation over $[0,2\pi]$ and such that 
		\begin{equation*}
			v_{0}\not\in \bigcup_{s>1/2} H^s_{\text{per}}(0,2\pi),
		\end{equation*}
    where the function space $H^s_{\text{per}}(0,2\pi)$ denotes the $2\pi$-periodic Sobolev space of order $s\geq0$,
        \begin{equation*}
            H^s_{\text{per}}(0,2\pi) = \Big\{\ v_{0} \in L^{2}(0,2\pi)\ ; \ \sum_{j\in\mathbb{Z}} (1+j^{2})^{s} |\widehat{f}(j)|^{2} < \infty \Big\}. 
        \end{equation*}
		Then, for irrational values of $t/2\pi$, the fractal dimension of the graph of the functions $\text{Re}\big(v(x,t)\big)$ and  $\text{Im}\big(v(x,t)\big)$ lies in $[1 + 2^{1-n}, 2-2^{1-n}]$.
\end{theorem}

Let us now consider the quasi-periodic problems and assume that $u_{0}$ is of bounded variation on $[0,2\pi]$ and such that
    \begin{equation*}
       z_{0}(\cdot) = e^{-i\theta x} u_{0}(\cdot) \not\in \bigcup_{s>1/2} H^s_{\text{per}}(0,2\pi).
    \end{equation*}

In the case of the second-order problem \eqref{Second order Quasi-Periodic Problem}, it follows that the fractal dimension of the real and imaginary parts of the solution at irrational times ($t/2\pi\not\in\mathbb{Q}$) is $3/2$. This is due to \eqref{Second order Correspondence} and the assumption that $\alpha_{2}\in\mathbb{Z}$ combined with Thoerem~\ref{Periodic fractalisation} applied to the solution $z(x,t)$ of the periodic problem \eqref{Second order Periodic Problem}.

For the solution $u(x,t)$ of the higher-order quasi-periodic problem \eqref{Quasi-Periodic Problem} recall from \eqref{Correspondence transformation inverse 1} that at a rational time $t_{r}=2\pi\frac{p}{q}$ it is given by
\begin{equation*} 
    u(x,t_{r}) = e^{-i(P(\theta) t_{r}  - \theta x)} \mathcal{T}_{s_{\theta} t_{r}} z(x,t_{r}),
\end{equation*}
where from Proposition~\ref{PPP Solution Representation Proposition} we have that
    \begin{equation*}
    z(x,t_{r}) = \prod_{m=2}^{n-1} \prod_{k=2}^{m} \mathcal{R}_{k} \Bigg(\alpha_{m+1}\pmqty{m+1\\k} \theta^{m+1-k} t_{r}\Bigg)v_{0}(x), \text{ with } v_{0}(x) = \prod_{\ell = 2}^{n} \mathcal{R}_{\ell} (\alpha_{\ell} t_{r})z_{0}(x).
    \end{equation*}
   Note that because $v_{0}$ is a finite linear combination of periodic translations of $z_{0}$, it satisfies the same hypothesis as $z_{0}$. If $n=3$, then $z(x,t_{r})$ becomes
\begin{equation*}
    z(x,t_{r}) = \mathcal{R}_{2}\Big(2\pi \frac{\alpha_{3}p}{q} \theta\Big) v_{0}(x), \text{ with } v_{0}(x) = \mathcal{R}_{2}\Big(2\pi\frac{\alpha_{2} p}{q}\Big)\mathcal{R}_{3}\Big(2\pi \frac{\alpha_{3} p}{q}\Big) z_{0}(x).
\end{equation*}
Assuming that $\theta\not\in\mathbb{Q}$, it follows that the time $2\pi \frac{a_{3} p}{q} \theta$ is an irrational time, and therefore by Theorem~\ref{Periodic fractalisation} the fractal dimension of $\Re(z(x,t_{r}))$ and $\Im(z(x,t_{r}))$ is $3/2$. Hence, regarding the solution of the quasi-periodic problem \eqref{Quasi-Periodic Problem}, we deduce that when $n=3$ the fractal dimension of $\Re(u(x,t_{r}))$ and $\Im(u(x,t_{r}))$ is also $3/2$. For $n\geq4$, we do not know explicitly the fractal dimensions. Indeed, the action on $v_{0}$ of each of the operators
    \begin{equation*}
        \mathcal{R}_{k} \Bigg(\alpha_{m+1}\pmqty{m+1\\k} \theta^{m+1-k} t_{r}\Bigg)
    \end{equation*}
might change the fractal dimension according to Theorem~\ref{Periodic fractalisation}. More importantly, such an action could result in functions that are not necessarily of bounded variation which in turn restricts the repeated application of Theorem~\ref{Periodic fractalisation}.

\section{Numerical examples}\label{Numerical Examples}

In this section, we give numerical examples that illustrate our main results. We consider the quasi-periodic problem \eqref{Quasi-Periodic Problem} with $P(\lambda) = \lambda^{n}$, $n\geq 3$ and piece-wise constant initial data at time $t=0$, see Figure~\ref{F initial},
given by 
\begin{equation}
    \label{initial example}
    u_{0}(x) = 
    \begin{cases}
    1, \quad x\in \big(\frac{\pi}{2}, \frac{3\pi}{2}\big),\\
    0, \quad x\in (0,2\pi)\setminus \big(\frac{\pi}{2}, \frac{3\pi}{2}\big).
    \end{cases}
\end{equation}
For fixed $n\geq 3$ and $t>0$, the solution is given by the eigenfunction expansion 
\begin{equation}
    \label{solution example}
    u(x,t) = \frac{i}{2\pi}
    \sum_{j\in\mathbb{Z}} 
    \frac{e^{-\frac{3\pi i}{2}(j+\theta)} - e^{-\frac{\pi i }{2}(j+\theta)}}{j + \theta} \exp(i \big((j+\theta)x - (j+\theta)^{n} t\big)),
\end{equation}
with respect to the orthonormal basis $\{\phi_{j}\}_{j\in\mathbb{Z}}$ given by \eqref{quasi-periodic basis}.
\begin{figure}[H]
\centering\includegraphics[width=0.35\textwidth]{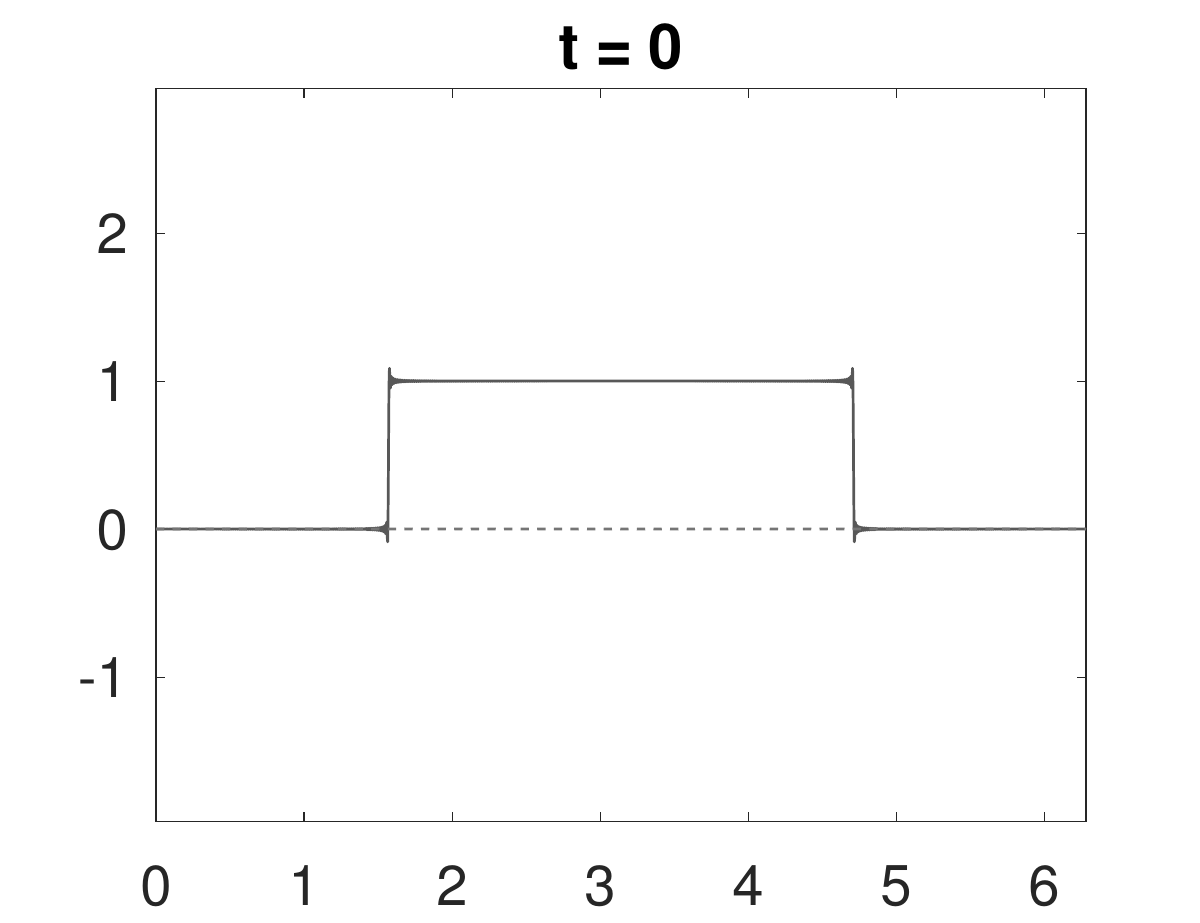}
 \caption{Real (solid) and imaginary (dashed) profiles of the piece-wise constant initial function \eqref{initial example}.}
	\label{F initial}
 \end{figure}

Using Octave, we plot the real and imaginary parts of the solution \eqref{solution example} for $n=3$, $4$, $5$ at the rational time $t= \frac{2\pi}{3}$ by truncating the series and taking around 1000 terms. In Figure \ref{F-n revivals}, we consider $\theta = \frac{1}{4}$ and observe piece-wise constant profiles. This confirms the first part of Theorem~\ref{Main theorem} which implies that the revival effect appears for rational values of $\theta\in(0,1)$. On the other hand for irrational values, the second part of Theorem~\ref{Main theorem} implies that the revival effect breaks down at rational times and the solution becomes continuous on $[0,2\pi]$ and such that $e^{i2\pi\theta}u(0,t) = u(2\pi,t)$, which is exactly what we observe in Figure \ref{F-n no revivals}. In particular, in Figure \ref{F-n no revivals} we see that the solution profiles have the form of continuous but nowhere differentiable functions, indicating that the fractalisation effect occurs at rational times when $\theta$ is irrational. The latter is true at least for $n=3$, with the fractal dimension being $3/2$, as we deduced in Section~\ref{Second-order case and fractal dimensions}. In contrast to the higher-order models, the second-order case exhibits revivals at rational times for any choice of $\theta\in (0,1)$. This is illustrated in Figure~\ref{F-n=2} where we plot the real and imaginary parts of \eqref{solution example}, when $n=2$ and $t = \frac{2\pi}{3}$, and we observe piece-wise constant profiles for both $\theta = \frac{1}{4}$ and $\theta=\frac{\sqrt{2}}{4}$, as expected by \eqref{Second order Correspondence}.

\begin{figure}[H]
	\hspace{-1cm}
	\begin{minipage}[b]{0.45\linewidth}
		\centering	\includegraphics[width=0.73\textwidth]{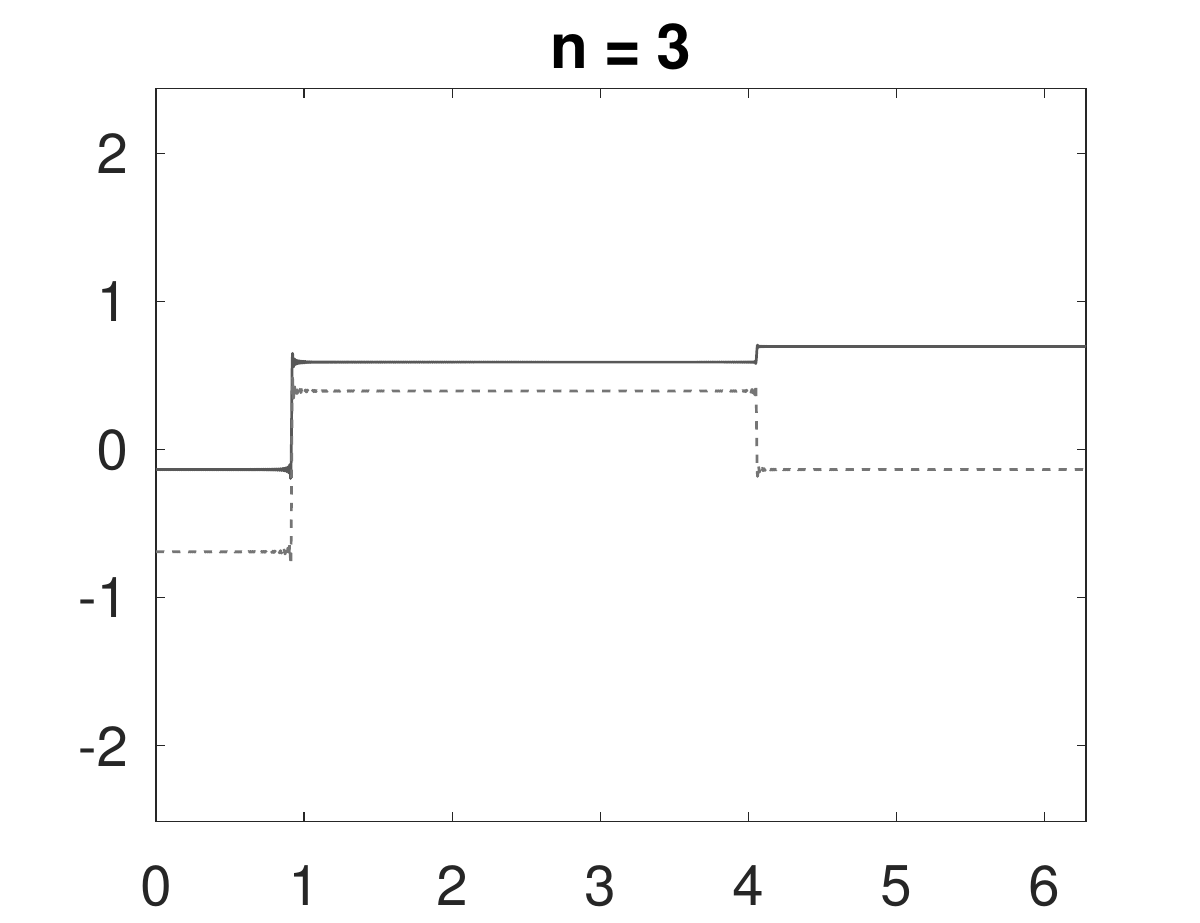}
	\end{minipage}
	\hspace{-2.2cm}
	\begin{minipage}[b]{0.45\linewidth}
		\centering	\includegraphics[width=0.73\textwidth]{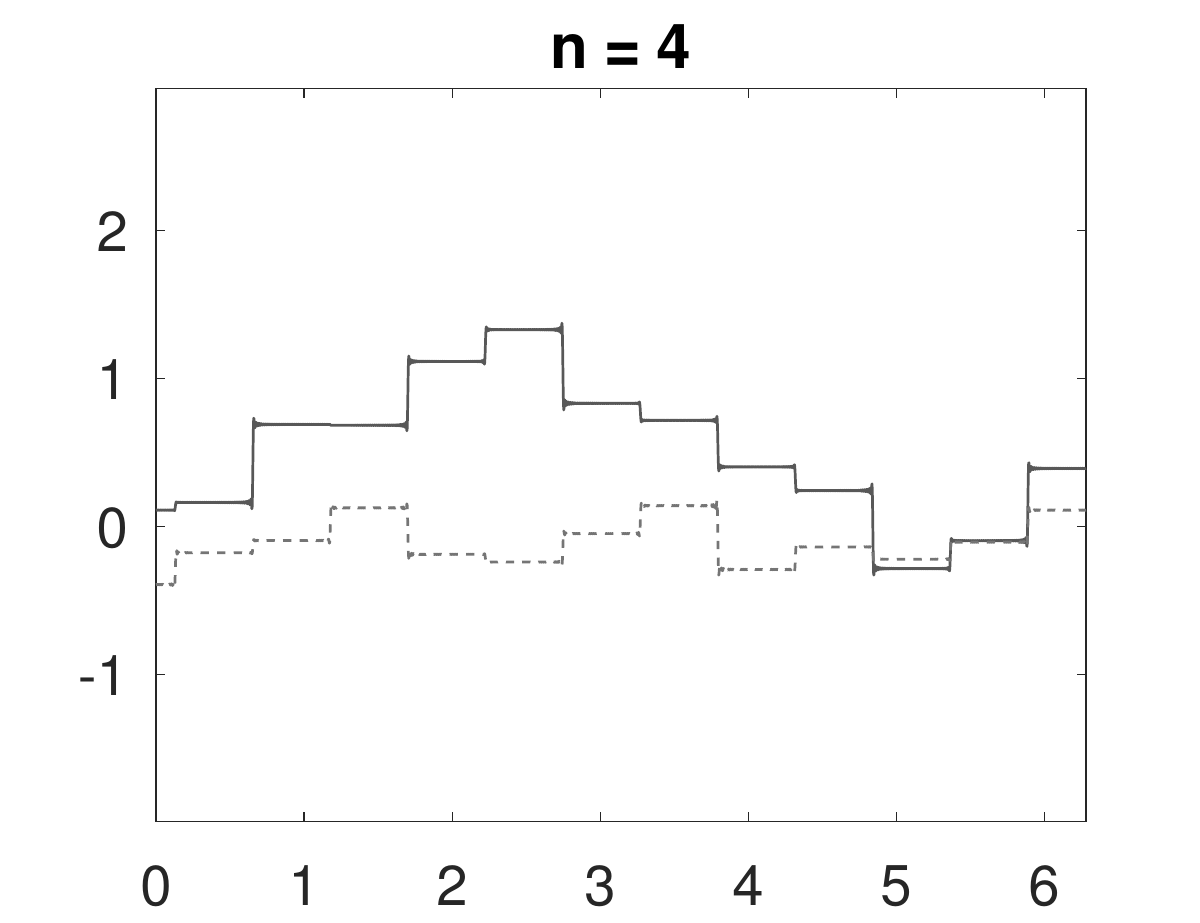}
	\end{minipage}
	\hspace{-2.2cm}
	\begin{minipage}[b]{0.45\linewidth}
		\centering \includegraphics[width=0.73\textwidth]{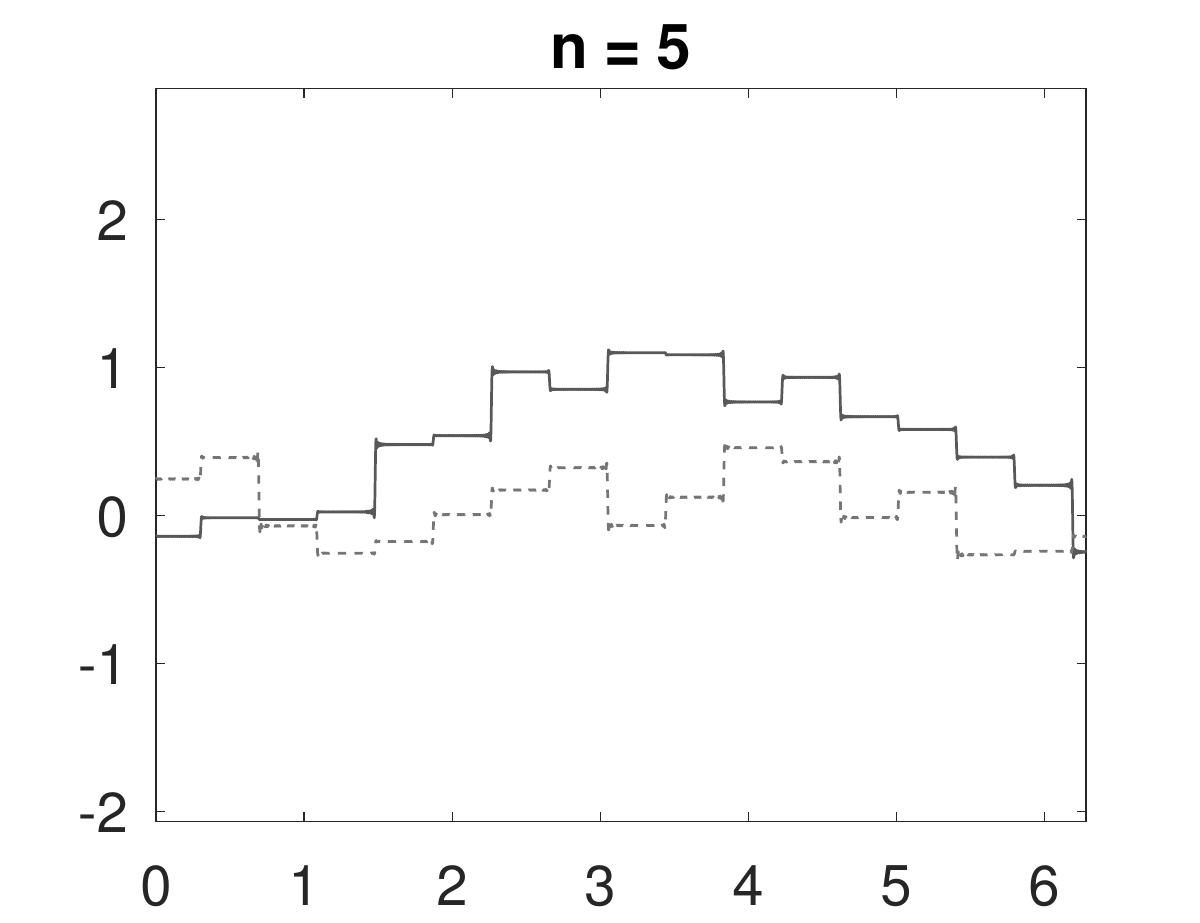}
	\end{minipage}
 \caption{Real (solid) and imaginary (dashed) profiles of the solution \eqref{solution example} at rational time $t=\frac{2\pi}{3}$ when $\theta = \frac{1}{4}$ and $n=3$, $4$, $5$.}
 \label{F-n revivals}
\end{figure}

\begin{figure}[H]
	\hspace{-1cm}
	\begin{minipage}[b]{0.45\linewidth}
		\centering	\includegraphics[width=0.73\textwidth]{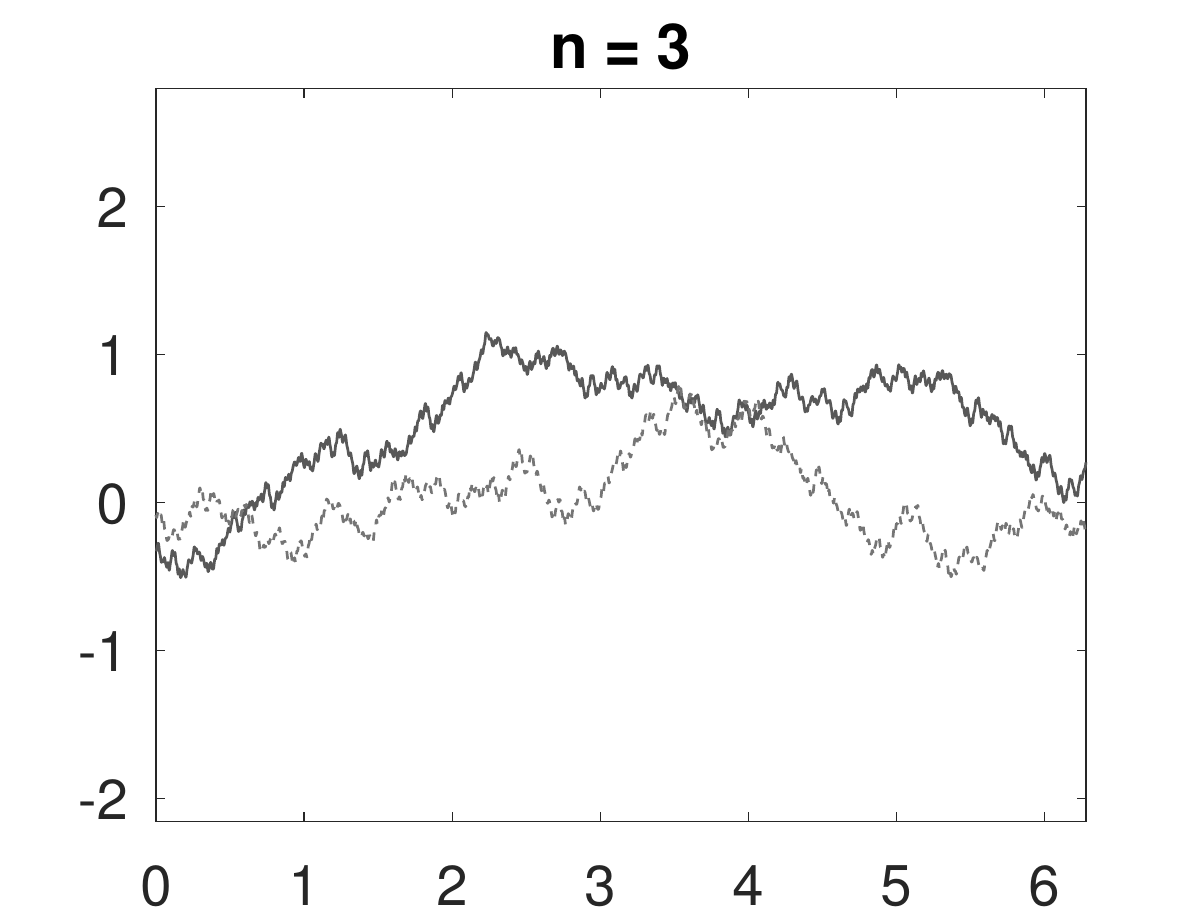}
	\end{minipage}
	\hspace{-2.2cm}
	\begin{minipage}[b]{0.45\linewidth}
		\centering	\includegraphics[width=0.73\textwidth]{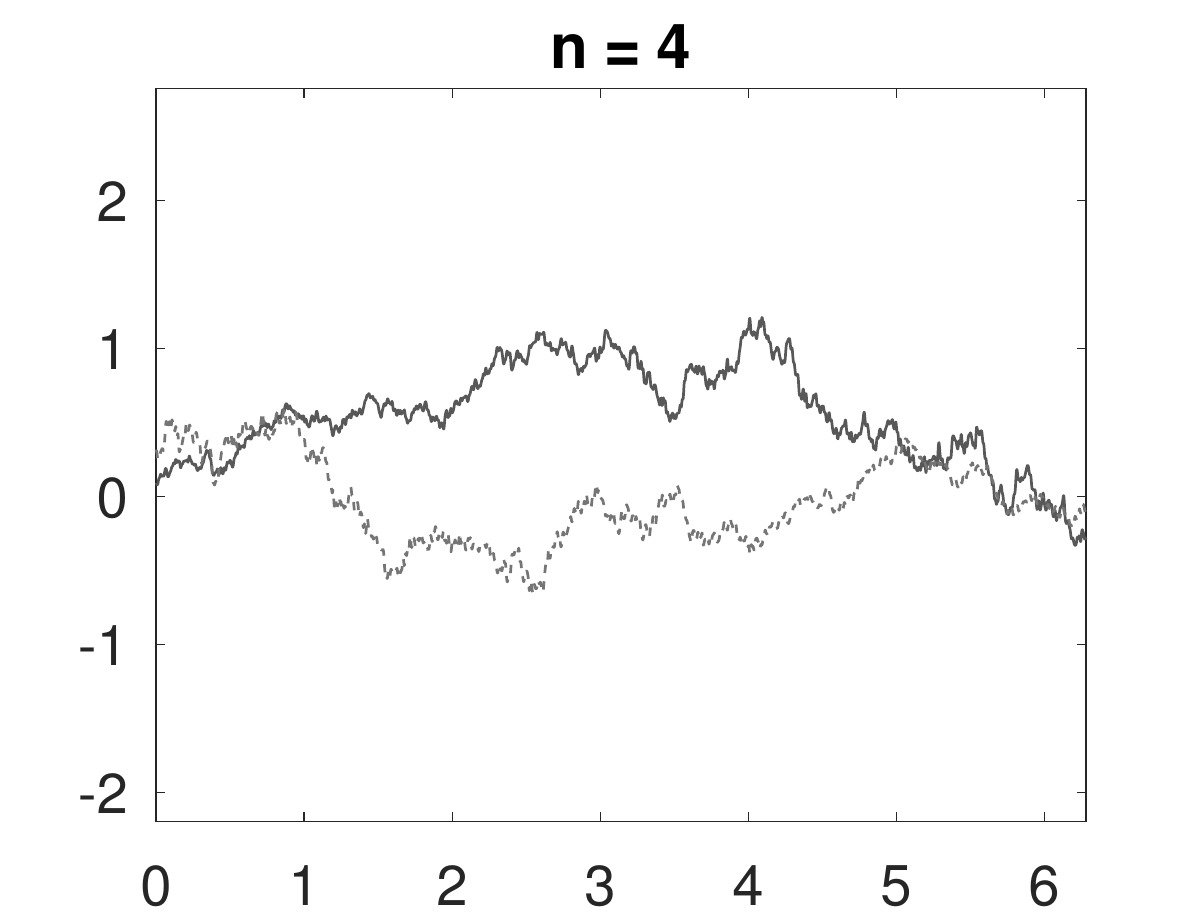}
	\end{minipage}
	\hspace{-2.2cm}
	\begin{minipage}[b]{0.45\linewidth}
		\centering \includegraphics[width=0.73\textwidth]{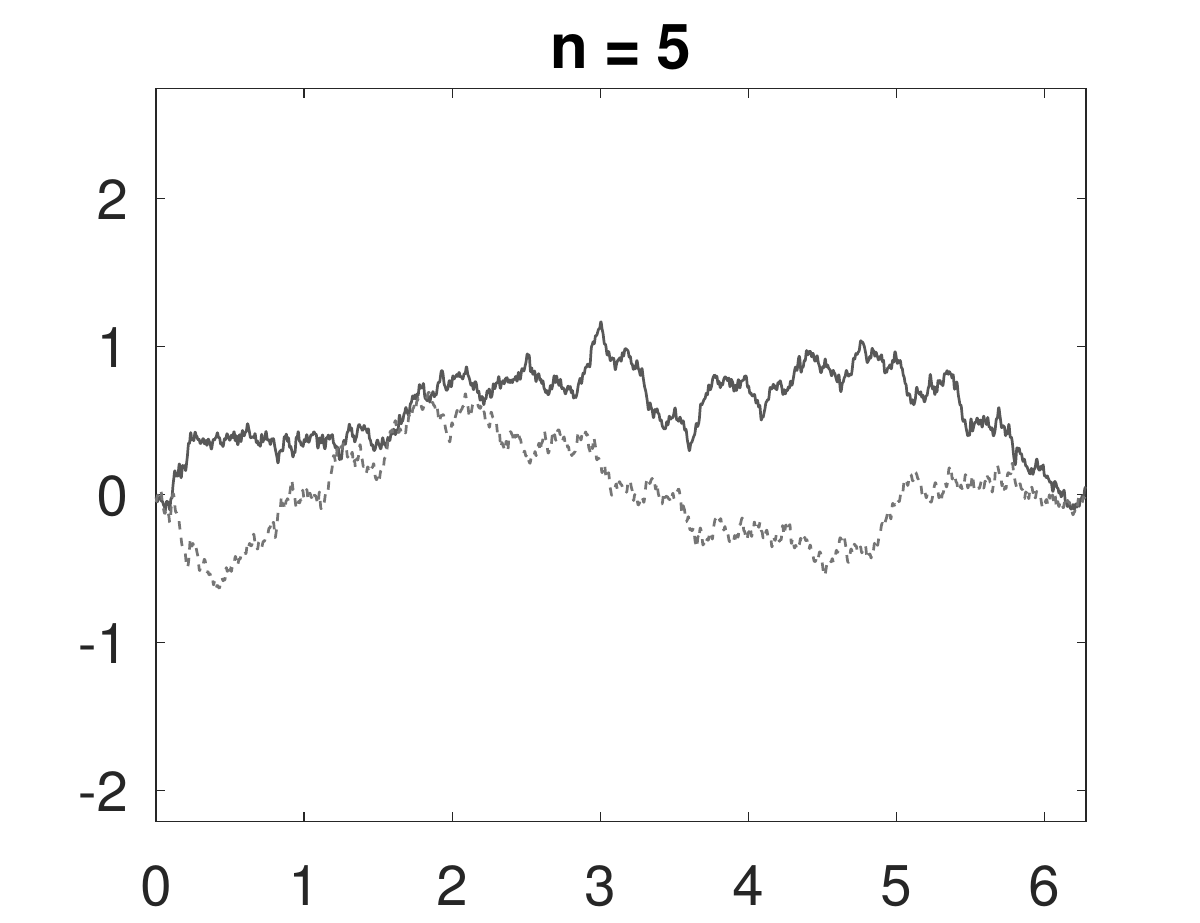}
	\end{minipage}
 \caption{Real (solid) and imaginary (dashed) profiles of the solution \eqref{solution example} at rational time $t=\frac{2\pi}{3}$ when $\theta = \frac{\sqrt{2}}{4}$ and $n=3$, $4$, $5$.}
 \label{F-n no revivals}
\end{figure}

\begin{figure}[H]
	\begin{minipage}[b]{0.45\linewidth}
		\centering	\includegraphics[width=0.73\textwidth]{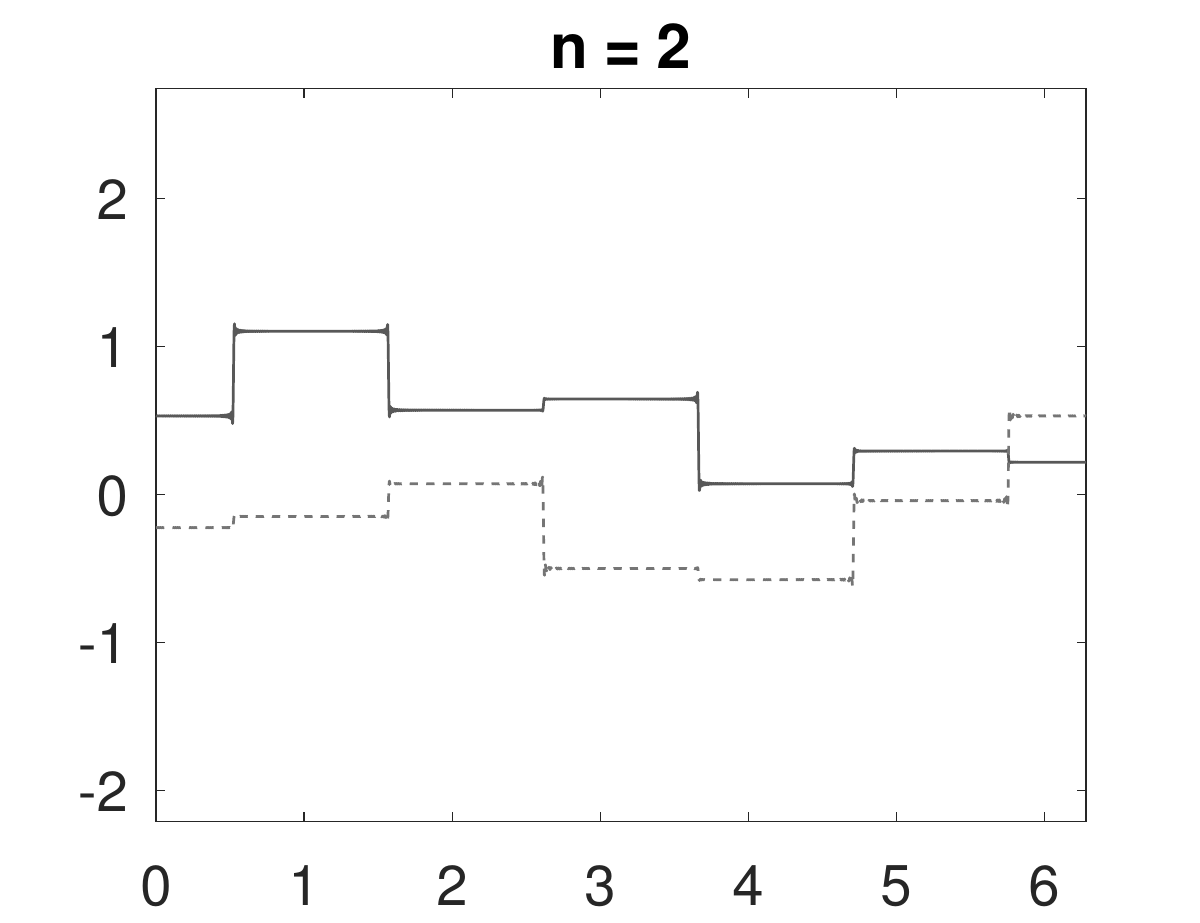}
	\end{minipage}
	\begin{minipage}[b]{0.45\linewidth}
		\centering	\includegraphics[width=0.73\textwidth]{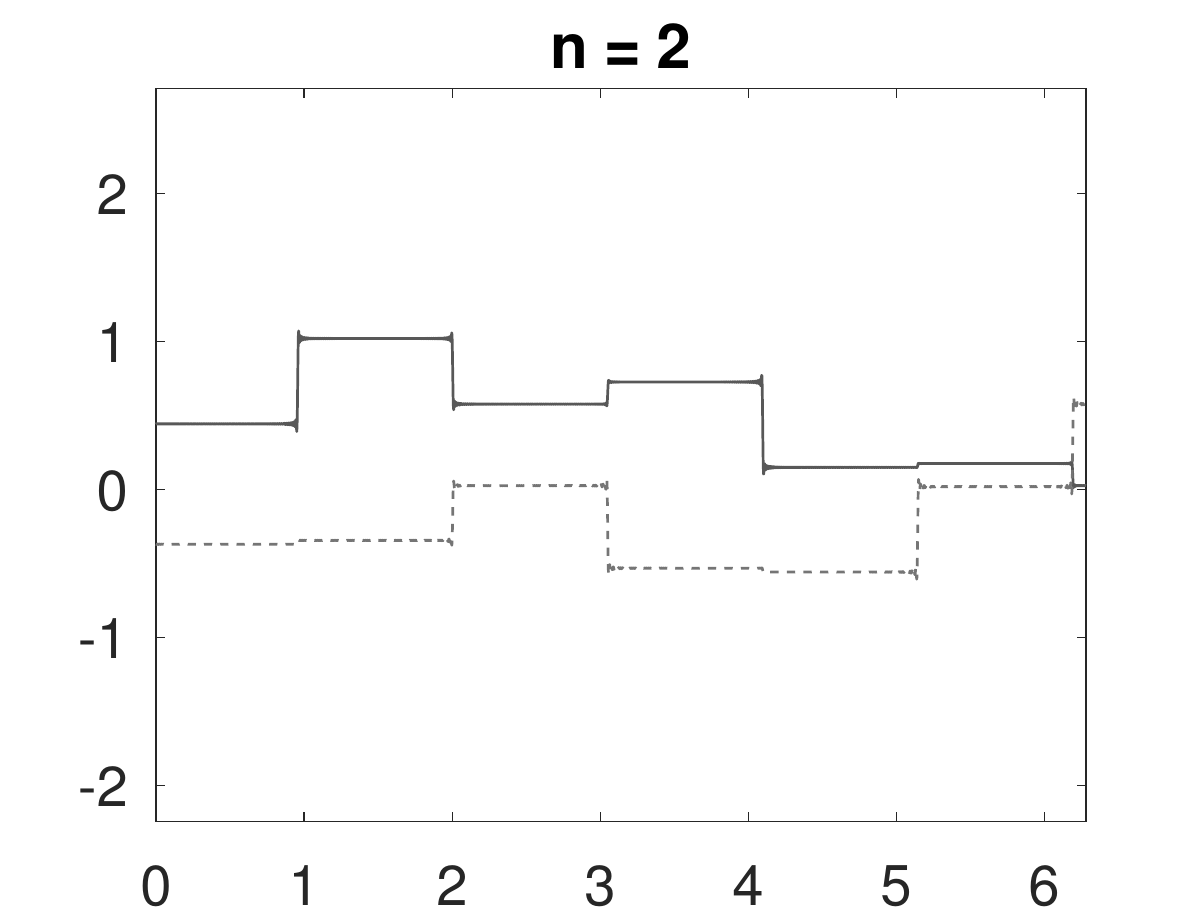}
	\end{minipage}
 \caption{Real (solid) and imaginary (dashed) profiles of the solution \eqref{solution example} at rational time $t=\frac{2\pi}{3}$ when $n=2$ and $\theta =\frac{1}{4}$ or $\theta = \frac{\sqrt{2}}{4}$, respectively.}
 \label{F-n=2}
\end{figure}

\section{Two non-linear equations}\label{Revivals in non-linear equations}

The revival and fractalisation phenomena have also been considered in the non-linear regime. In  \cite{erdogan2013talbotpaper, erdougan2013global}, Erdo{\u{g}}an and Tzirakis examined the Talbot effect under periodic boundary conditions on $[0,2\pi]$ in the context of the cubic non-linear Schr\"{o}dinger (NLS) equation 
\begin{equation*}
	\partial_{t}u(x,t) = i\partial_{x}^{2}u(x,t) + i |u(x,t)|^{2} u(x,t)
\end{equation*}
and the Korteweg–de Vries (KdV) equation
\begin{equation*}
	\partial_{t}u(x,t) = - \partial_{x}^{3}u(x,t) - 2 u(x,t) \partial_{x}u(x,t),
\end{equation*}
respectively. The results in \cite{erdogan2013talbotpaper} and \cite{erdougan2013global} showed that for initial data of bounded variation, the difference between the solution and the linear evolution, at any positive time, is more regular than the initial function. Thus, in both cases, the revival/fractalisation dichotomy is characterised in terms of change in the regularity of the solution at rational/irrational times. Numerical evidence on the Talbot effect for both equations can be found in the papers of  \cite{chen2013dispersion} and \cite{chen2014numerical} by Chen and Olver. 

Similar to the FLS equation, the cubic NLS equation is invariant under the Galilean transformation 
\begin{equation*}
	u(x,t) = e^{-i\theta^{2} t} e^{i\theta x} \mathcal{T}_{2\theta t} z(x,t).
\end{equation*}
Therefore, we can extend the results in \cite{erdogan2013talbotpaper} from the periodic to the quasi-periodic setting. 

\begin{theorem}
	\label{Quasi Periodic NLS Revival}
	Fix $\theta \in (0,1)$ and consider the quasi-periodic problem for the NLS equation
 \begin{equation}
	\label{Quasi-periodic NLS}
	\begin{aligned}
		&\partial_{t}u(x,t) = i \partial_{x}^{2}u(x,t)+i|u(x,t)|^{2}u(x,t), \quad u(x,0) = u_{0}(x),\\
		&e^{i2\pi \theta} u(0,t) = u(2\pi,t),\quad e^{i2\pi \theta}\partial_{x} u(0,t) = \partial_{x}u(2\pi,t).
	\end{aligned}
\end{equation}
 Assume that $u_{0}$ is of bounded variation on $[0,2\pi]$. Then, we have the following.
	\begin{itemize}
		\item [(i)] For rational values of $t/2\pi$, the solution $u(x,t)$ is a bounded function with at most countably many discontinuities.
        \item[(ii)] If $t/2\pi$ is an irrational number, then the solution $u(x,t)$ is a continuous function of $x\in[0,2\pi]$ and such that $e^{i2\pi\theta} u(0,t) = u(2\pi,t)$.
	\end{itemize}
\end{theorem}
\begin{proof}
Due to the Galilean invariance of the cubic NLS equation, the solution to the quasi-periodic problem \eqref{Quasi-periodic NLS} is given by 
\begin{equation*}
    u(x,t) = e^{-i\theta^{2} t} e^{i\theta x} \mathcal{T}_{2\theta t} z(x,t),
\end{equation*}
where $z(x,t)$ solves the cubic NLS with periodic boundary conditions on $[0,2\pi]$
     \begin{equation*}
	\begin{aligned}
		&\partial_{t}z(x,t) = i \partial_{x}^{2}z(x,t)+i|z(x,t)|^{2}z(x,t), \quad z(x,0) = z_{0}(x),\\
		&z(0,t) = z(2\pi,t),\quad \partial_{x} z(0,t) = \partial_{x}z(2\pi,t),
	\end{aligned}
\end{equation*}
and with initial value $z_{0}(x) = e^{-i\theta x} u_{0}(x)$. Since $u_{0}$ is a function of bounded variation, the function $z_{0}$ is also of bounded variation as the product of the non-zero continuous function $e^{-i\theta x}$ and $u_{0}$. Moreover, from \cite[Theorem~1]{erdogan2013talbotpaper}, we know that the statement of the theorem holds for the solution $z(x,t)$ to the periodic problem. Thus, the theorem holds for $u(x,t)$ as well. 
\end{proof}

If we further assume that the function $z_{0}(x) = e^{-i\theta x} u_{0}(x)$ does not belong to any of the Sobolev space $H^{s}_{\text{per}}(0,2\pi)$, for any $s>1/2$, then we can deduce that either the real or the imaginary part of the solution $u(x,t)$ to the quasi-periodic problem \eqref{Quasi-periodic NLS} has fractal dimension $3/2$ at irrational times, since the same is true of the solution to the periodic problem, see \cite[Theorem 1]{erdogan2013talbotpaper}.

Compared to the quasi-periodic cubic NLS, the previous trick does not work in the case of the quasi-periodic KdV equation 
\begin{equation}
	\label{KdV quasi periodic}
	\begin{aligned}
		&\partial_{t}u(x,t) = - \partial_{x}^{3}u(x,t) - 2u(x,t) \partial_{x}u(x,t), \quad u(x,0) = u_{0}(x),\\
		&e^{i2\pi\theta}\partial_{x}^{m}u(0,t) = \partial_{x}^{m}u(2\pi, t), \quad m=0,1,2, \quad \theta\in (0,1).
	\end{aligned}
\end{equation}
Therefore, we describe below two possible ideas in order to examine the revival and fractalisation effects in this case.

Under the transformation
\begin{equation*}
	u(x,t) = e^{i(\theta^{3} t + \theta x)} w(x,t),
\end{equation*}
the quasi-periodic problem \eqref{KdV quasi periodic} is converted to the periodic problem 
\begin{equation}
	\label{KdV Transformed problem}
	\begin{aligned}
		&\partial_{t}w(x,t) =\big( - \partial_{x}^{3} -3i\theta \partial_{x}^{2} + 3\theta^{2} \partial_{x} \big) w(x,t) + N(x,t,w,\partial_{x} w) , \quad w(x,0) = e^{-i\theta x} u_{0}(x),\\
		&\partial_{x}^{m}w(0,t) = \partial_{x}^{m}w(2\pi, t), \quad m=0,1,2,
	\end{aligned}
\end{equation}
where the non-linear term is given by
\begin{equation*}
	N(x,t,w,\partial_{x}w) = -e^{i(\theta^{3}t - \theta x)}\big(2i\theta w^{2}(x,t) + 2w(x,t) \partial_{x}w(x,t)\big).
\end{equation*}

First, we remark that, to our knowledge, there is no result that indicates the well-posedness in $L^{2}(0,2\pi)$ of either of the two initial-boundary value problems \eqref{KdV quasi periodic} and \eqref{KdV Transformed problem}. Therefore, a first direction would be to analyse this aspect. We suspect that the framework provided by Ruzhanski and Tokmagambeto in  \cite{ruzhansky2016nonharmonic} should be valuable for the analysis of the quasi-periodic problem \eqref{KdV quasi periodic}, whereas for the corresponding periodic problem \eqref{KdV Transformed problem} it might be feasible to adopt some of the arguments of Miyaji and Tsutsumi in \cite{MIYAJI20171707} on the analysis of the periodic third-order Lugiato-Lefever equation in conjunction with the more recent work on the Talbot effect for the same equation \cite{cho2023talbot} by Cho, Kim and Seo.

Finally, if we assume that there exists a well-posed setting for either \eqref{KdV quasi periodic} or \eqref{KdV Transformed problem}, then a natural question that arises is if the revival phenomenon breaks whenever $\theta$ is irrational, similar to the corresponding linear problems. In other words, if we fix $\theta$ to be an irrational number in $(0,1)$, then is there a smoothing effect on the solution of either \eqref{KdV quasi periodic} or \eqref{KdV Transformed problem} at rational times ($t/2\pi \in\mathbb{Q}$)?

\section{Conclusion}\label{Conclusion}
The main goal of this work was to examine the revival effect in the context of time-evolution problems with quasi-periodic boundary conditions. Theorem~\ref{Main theorem} establishes that for a piece-wise continuous initial condition in $L^{2}(0,2\pi)$, the quasi-periodic problem \eqref{Quasi-Periodic Problem} of order $n\geq 3$ supports the revival effect at rational times only when the boundary parameter $\theta$ is a rational number. Otherwise, for irrational values of $\theta$, the solution at rational times becomes continuous and quasi-periodic, so the revival (of the initial jump discontinuities) breaks down.

The proof of Theorem~\ref{Main theorem} was obtained as a consequence of combining  Lemma~\ref{Correspondece lemma} and Theorem~\ref{Main theorem 2}. Lemma~\ref{Correspondece lemma} allows to study the revival effect in the quasi-periodic problem \eqref{Quasi-Periodic Problem} by means of the periodic problem \eqref{Periodic problem}. In turn, Theorem~\ref{Main theorem 2} illustrates that the periodic problem \eqref{Periodic problem} supports the revival property if and only if $\theta$ is rational. The latter also implies that periodic linear dispersive PDEs with polynomial dispersion relation of order $n\ge 3$ can exhibit revivals only when the coefficients of the dispersion relation are all rational. 

Further consequences of our approach included the verification and the extension in the quasi-periodic setting and for any $\theta \in (0,1)$ of the Talbot effect in the free linear Schr\"{o}dinger and the cubic non-linear Sch\"{o}dinger equations, respectively. 
\vspace{1.cm}

\small{\textbf{Acknowledgements} The author thanks Lyonell Boulton and Beatrice Pelloni for their useful comments and suggestions on previous versions of the manuscript. The work
was funded by EPSRC through Heriot-Watt University support for Research Associate positions, under the Additional Funding Programme for the Mathematical Sciences. Also, part of this research was conducted during the author's PhD studies which were supported by the Maxwell Institute Graduate School in Analysis and its Applications, a Centre for Doctoral Training funded by EPSRC (grant EP/L016508/01), the Scottish Funding Council, Heriot-Watt University and the University of Edinburgh.}

\printbibliography

\end{document}